\documentclass[12pt]{amsart}
\usepackage{epsfig,color}

\makeatother

\theoremstyle{definition}

\def\fnum{equation}
\newtheorem{Thm}[\fnum]{Theorem}
\newtheorem{Cor}[\fnum]{Corollary}

\newtheorem{Lem}[\fnum]{Lemma}

\newtheorem{Pro}[\fnum]{Proposition}

\numberwithin{equation}{section}

\newcommand{\nn}{{\bf{n}}}

\newcommand{\dist}{{\text {dist}}}

\newcommand{\cK}{{\mathcal{K}}}

\newcommand{\Hess}{{\text {Hess}}}

\def\RR{{\bold R}}

\def\SS{{\bold S}}

\newcommand{\dv}{{\text {div}}}
\newcommand{\e}{{\text {e}}}

\newcommand{\cH}{{\mathcal{H}}}
\newcommand{\cN}{{\mathcal{N}}}
\newcommand{\cO}{{\mathcal{O}}}
\newcommand{\cU}{{\mathcal{U}}}
\newcommand{\cM}{{\mathcal{M}}}

\newcommand{\eqr}[1]{(\ref{#1})}

\title[Wandering singularities]{Wandering singularities}
\author[]{Tobias Holck Colding}%
\address{MIT, Dept. of Math.\\
77 Massachusetts Avenue, Cambridge, MA 02139-4307.}
\author[]{William P. Minicozzi II}%
 
\thanks{The  authors
were partially supported by NSF Grants DMS 1812142 and DMS 1707270.}

\email{colding@math.mit.edu  and minicozz@math.mit.edu}

\begin{document}

\maketitle

\begin{abstract}
Parabolic geometric flows are smoothing for short time however, over long time, singularities are typically unavoidable,  can be very nasty and may be impossible to classify.  The idea of  \cite{CM6} and here is that, by bringing in the dynamical properties of the flow, we obtain also smoothing for large time for generic initial conditions. When combined with \cite{CM1}, this shows, in an important special case, the singularities are the simplest possible.

 The question of the dynamics  of a singularity has two parts.  One is:   What are the dynamics near a singularity?  The second is:  What is the long time behavior?  That is, if the flow leaves a neighborhood of a singularity, can it return at a much later time?  The first question was addressed in \cite{CM6} and the second here.
 
 Combined with \cite{CM1}, \cite{CM6}, we show that all other closed singularities than the (round) sphere have a neighborhood where ``nearly every'' closed hypersurface leaves under the flow   and never returns, even  to a dilated, rotated or translated copy of the singularity.  In other words, it {\it{wanders}} off.  In contrast, by Huisken,  any closed hypersurface  near a  sphere remains close to a dilated or translated copy of the sphere at each time.  
\end{abstract}

\section{Introduction}

The mean curvature flow, or MCF, is the negative gradient flow of volume on the space of submanifolds.  To understand the flow, the key  is to understand the singularities.  
A neighborhood of a singular point for the flow is modeled by its blow-up. Blow-ups, or tangent flows,  are  shrinkers \cite{H2}, \cite{I}, \cite{W}.  A one parameter family of submanifolds $M_t$ flowing by the MCF is   a shrinker  if it evolves by rescaling with $M_t=\sqrt{-t}\,M_{-1}$.  The simplest shrinkers are round spheres and cylinders, but  many other exotic shrinkers are known to exist.

Suppose that $M_t \subset \RR^{n+1}$ is a one-parameter family of closed hypersurfaces flowing by MCF.  We would like to analyze the flow near a singularity.    If we reparametrize and rescale the flow, then we get a solution to the rescaled MCF equation which is the negative gradient flow for the Gaussian volume ($F$-functional)
\begin{align}
	F(\Sigma) = 	\int_{\Sigma} \e^{ - \frac{|x|^2}{4} } \, .
\end{align}
  Here the gradient is with respect to the weighted inner product on the space of normal variations.    The fixed points of the rescaled MCF, or equivalently the critical points of the $F$-functional, are shrinkers.   This rescaling   turns the question of the dynamics of the MCF near a singularity into a question of the dynamics near a fixed point for the rescaled flow.

Existing examples of exotic shrinkers,  \cite{A},  \cite{Ch}, \cite{KKM}, \cite{K}, \cite{Nu},  
   suggest   that  a complete classification of shrinkers 
is impossible for $n>1$.     This lack of possible classification makes the question of whether these exotic singularities   occur generically even more relevant.  
  In \cite{CM1}, we showed that the only smooth stable shrinkers are spheres, planes, and generalized cylinders (i.e., $\SS^k \times \RR^{n-k}$).    In particular, the round sphere is the only closed stable singularity.   A closed shrinker is said to be stable if, modulo translations and dilations, the second derivative of the $F$-functional is non-negative for all variations, see \cite{CM1}.
    We show here that, in a suitable sense, ``nearly every'' hypersurface in a neighborhood of a closed unstable shrinker is wandering or, equivalently, non-recurrent.\footnote{A dynamical system is 
  {\it dissipative} if it has a wandering set of non-zero measure. This is  the opposite of a conservative system, for which the ideas of the Poincar\'e recurrence theorem apply. Intuitively,   if a portion of the phase space ``wanders away'' during normal time-evolution, and is never visited again, then the system is dissipative.    The notion of wandering sets  was introduced by Birkhoff in 1927.}

  Let $\Sigma \subset \RR^{n+1}$  be a smooth closed embedded hypersurface with unit normal $\nn$.  We will identify nearby hypersurfaces with  functions on $\Sigma$ by identifying a function with its graph.  Namely, let
  $E$ be the Banach space of $C^{2,\alpha}$ functions on  $\Sigma$ and let  $\Upsilon$ be the map from $E$ to subsets of $\RR^{n+1}$ that takes $u \in E$ to its normal exponential graph $\Upsilon (u)$
\begin{equation}
	 \Upsilon (u)  = \{ p + u(p) \, \nn (p) \, | \, p \in \Sigma \} \, .
\end{equation}  
Since $\Sigma$ is closed and embedded, 
there is a neighborhood $\cU$ of $0$ in $E$ where $\Upsilon$ is a bijection to a neighborhood $\hat{\cU} = \Upsilon (\cU)$ of $\Sigma$ in the space of $C^{2,\alpha}$ closed hypersurfaces.
When it is clear, we will   identify a function $u$ with its graph $\Upsilon (u)$.

The conformal linear group $\cH$ of $\RR^{n+1}$ is  generated by   rigid motions and  dilations.  This group acts naturally on subsets of $\RR^{n+1}$,  preserving the space of 
$C^{2,\alpha}$ closed hypersurfaces.  
However, not all elements  of  the orbit $\cH(\Sigma) = \bigcup_{g \in \cH} g(\Sigma)$ are graphs over $\Sigma$.

Our first main result is that a closed hypersurface that flows out of a neighborhood of a closed shrinker and its orbit under  $\cH$ can never return:
  
  \begin{Thm}	\label{t:lojas}
Let $\Sigma\subset \RR^{n+1}$ be any closed $n$-dimensional shrinker and $\cO$ an open neighborhood of $\Sigma$. There exist  open neighborhoods $\cU_1\subset \cU_2 \subset \cO$, such that if a rescaled MCF starts at a closed hypersurface $M$ in $\cU_1$ and leaves $\cH(\cU_2)$, then it cannot return to $\cH(\cU_1)$.
\end{Thm}

 Theorem  \ref{t:lojas} is a global result.  There are many flows, even gradient flows, where the conclusion fails and the flow returns infinitely often to a neighborhood of a fixed point.

 In \cite{CM6} we showed that a typical closed hypersurface near an unstable shrinker  leaves the orbit of a  neighborhood of the shrinker.   Combining this with Theorem \ref{t:lojas} gives:

\begin{Thm} 	\label{t:one}
Suppose that $\Sigma^n \subset \RR^{n+1}$ is a smooth closed embedded  shrinker, but is not a round sphere.  There exist  an open  neighborhood $\cO $ of $0\in E$ and  a  subset $ W$ of $\cO$ so that:
\begin{itemize}
\item There is a splitting $E= E_1 \oplus E_2$  with  $\dim (E_1) > 0$ 
so that  $W$ is contained in the graph $(x,u(x))$ of a continuous mapping     $u: E_2 \to E_1$.
\item  If $\Sigma' \in \Upsilon(\cO \setminus W)$, then the rescaled MCF starting at $\Sigma'$ leaves $\cO$ and its orbit  $\cH(\cO)$ under  $\cH$ and never returns.  
\end{itemize}
\end{Thm}

In contrast to Theorem \ref{t:one}, in a small neighborhood of the round sphere, all closed hypersurfaces are convex and, thus, all become extinct in a round sphere under the MCF by a result of Huisken, \cite{H1}.  However, the extinction point in space-time varies as the hypersurface changes.  Correspondingly, under the rescaled MCF, it may leave a neighborhood of the round sphere but does so near a translation or dilation of the sphere.   Similarly,  there are closed hypersurfaces near any shrinker that flow away but do so  trivially near a rigid motion or dilation of the   shrinker.  Unlike Theorem \ref{t:one}, this leads to no real change/improvement.

\section{Proof of no return}   \label{s:s1}

A one-parameter family of hypersurfaces  $M_t \subset \RR^{n+1}$  flows by mean curvature if
\begin{equation}
	\left( \partial_t x \right)^{\perp} = - H \, \nn  \, .
\end{equation}
Here $x$ is the position vector, 
   $v^{\perp}$ is the  normal part a vector $v$,
  and the mean curvature\footnote{With this convention, $H$ is $n/R$ on the $n$-sphere of radius $R$ in $\RR^{n+1}$.}   is   $H = \dv \,  \nn = \langle \nabla_{e_i} \nn , e_i \rangle$.  
The $e_i$'s are an orthonormal frame for the hypersurface and we sum over   repeated indices.  
The one-parameter family $\tilde M_t = M_{-\e^{-t}}/\sqrt{\e^{-t}}$ is a  rescaled mean curvature flow centered at the origin in space-time    satisfying
\begin{align}
	x_t = \left( \frac{ \langle x , \nn \rangle}{2} - H \right) \, \nn \, .
\end{align}
In this section,  $\Sigma^n\subset \RR^{n+1}$ is a closed shrinker, $M_t$ a MCF, and $\tilde{M}_t = \frac{ M_{-\e^{-t}}}{\sqrt{\e^{-t}}}$ a rescaled MCF.

For a dynamical system, a Lyapunov function is a   monotone quantity.  Since   rescaled MCF is the negative gradient flow for $F$,   $F$ is a Lyapunov function.  In fact, see \cite{CM1}, the entropy 
\begin{align}
\lambda (M)=\sup_{t_0>0,x_0\in \RR^{n+1}}F (t_0 \, M + x_0) 
\end{align}
is   a Lyapunov function for both  MCF and   rescaled MCF.    The entropy is more relevant than   $F$  when one studies generic properties since $\lambda$ is unchanged when a singularity is pushed off to a different point in space-time.

The no return first  uses   \cite{CM1} to show that if it does return modulo $\cH$, then $F$ is almost constant.  Once we have this, then we use   ideas of  Simon, \cite{S}, and  Schulze, \cite{Sc}, to show that   the   flow is nearly
static:

\begin{Pro}  \label{p:Fsmall}
Given a closed smooth shrinker $\Sigma^n\subset \RR^{n+1}$, there exist  constants $C$  and $\beta > 0$ and an open neighborhood $\cO_{\Sigma}$ of $\Sigma$, such that:\\
  If $\tilde M_t\subset \cO_{\Sigma}$ for $t_1\leq t\leq t_2$, then $\tilde M_{t}$ is a normal exponential graph over $\Sigma$ of $u_t$ and
\begin{align}
	\int_{\Sigma} \left| u (p, t_2) - u (p,t_1) \right| 
	\leq C\,  [F(t_1)-F(t_2)]^{\beta} \, .
\end{align}
\end{Pro}

The group of translations  and dilations transforms the flow $M_t$ into related mean curvature flows.  Namely, if
  $(x_0, t_0 , a) \in \RR^{n+1} \times \RR \times \RR^{+}$, then   $a \left(M_{a^{-2}\,(t_0+t)}+x_0\right)$ is    a MCF and
\begin{align}	\label{e:rescaled}
	\tilde M_t(x_0 , t_0 ,a) = \frac{a}{\sqrt{\e^{-t}}} \, \left( M_{a^{-2}\,(t_0- \e^{-t})}+x_0
	\right) 
\end{align}
is a rescaled MCF.
  Similarly, if $g$ is in the orthogonal group $O(n+1)$, then $g(M_t)$ is  a MCF.

The next proposition shows that if a rescaled MCF   starts near $\Sigma$ and  returns near   $\cH (\Sigma)$, then we can adjust by the group action to get a related flow that is nearly static.  
 Theorem \ref{t:lojas}  will follow  immediately.

\begin{Pro}	\label{p:comingback}
Given $\epsilon>0$, there exists a $\delta>0$, such that if $T_1 < T_2$, 
\begin{align}
	\tilde M_{T_1} {\text{ and }} b\, g( \tilde M_{T_2} + y_0)  \in B_{\delta}(\Sigma) 
\end{align}
  for some $b>0$, $y_0\in \RR^{n+1}$ and $g \in O(n+1)$, then  the rescaled MCF $g\, \tilde M_t(   y_0 , 1 - b^2\e^{-T_2} ,  b )$ 
is $\epsilon$ close to $\Sigma$ for all $t \in \left[ -\log \left\{ 1 - b^2 (\e^{-T_2} - \e^{-T_1}) \right\} ,0 \right]$.  
\end{Pro}

\begin{proof}
For simplicity, we will assume that $g$ is the identity, so that 
 $M_{-\e^{-T_1}}/\sqrt{\e^{-T_1}}$ and $b\,( M_{-\e^{-T_2}}/\sqrt{\e^{-T_2}} + y_0)$  are $\delta$-close to $\Sigma$. 
The general case follows with obvious modifications.

  Consider   the rescaled flow $\bar M_t = \tilde M_t(x_0 , t_0 ,a)$ given by \eqr{e:rescaled} with 
  \begin{align}
   x_0 = y_0 , \,  t_0 = 1 - b^2 \, \e^{-T_2} \, ,  {\text{ and }} a = b   \, .
  \end{align}
  Time $t$ for the new rescaled flow corresponds to time 
  \begin{align}
  	a^{-2}\,(t_0- \e^{-t}) =  b^{-2} \left(  1- \e^{-t} \right)    -   \e^{-T_2}   
\end{align}
 for the MCF.  In particular,   $0$ corresponds to time $- \e^{-T_2}$ and  $\bar M_0$ is $\delta$-close to $\Sigma$ since
  \begin{align}	\label{e:Mzero}
  	\bar M_0 &= b\,( M_{-\e^{-T_2}}/\sqrt{\e^{-T_2}} + y_0) \, .
  \end{align}
  Moreover, time $\bar{T} = -\log \left\{ 1 - b^2 (\e^{-T_2} - \e^{-T_1}) \right\}$ for the new rescaled flow corresponds to $\e^{-T_1}$ for the MCF and, thus,
  $\bar M_{\bar{T}}$ is a translation and dilation of $M_{-\e^{-T_1}}$.   
  
  From \cite{CM1}, $
F(\Sigma)=\lambda(\Sigma)$  and, since $\Sigma$ is compact,   lemma $7.10$ in \cite{CM1}
gives
 \begin{equation}
 	\lambda (M_{-\e^{-T_1}}) =  \lambda (M_{-\e^{-T_1}}/\sqrt{\e^{-T_1}}) \leq F(\Sigma) +  \kappa (\delta) \, , 
\end{equation} 
 where $\kappa$ is   continuous  with $\kappa (0) = 0$.  Since $\lambda$ is invariant under dilations and translations, 
\begin{align}
	F(\bar M_{\bar{T}}) \leq \lambda (\bar M_{\bar{T}}) = \lambda (M_{-\e^{-T_1}}) \leq F(\Sigma) + \kappa (\delta) \, .
\end{align}
Similarly,   \eqr{e:Mzero} yields $F(\bar M_0) > F(\Sigma) - \kappa (\delta)$.  Since $F$ is monotone for rescaled MCF,  
\begin{align}	\label{e:Fclose}
	\sup_{t \in [\bar{T} , 0]} \, \left| F( \Sigma) - F(\bar{M}_t ) \right| \leq 2 \, \kappa (\delta) \, .
\end{align}
Since $F$ does not change much, Proposition \ref{p:Fsmall}  gives that $\bar{M}_t$ does not change much either.  To be precise, 
  fix   $\epsilon_0 < \frac{\epsilon}{2}$ small enough that Proposition \ref{p:Fsmall}  applies in $B_{\epsilon_0}(\Sigma)$.  Since the flow is continuous and $\Sigma$ is a fixed point,  there exists
  $\delta_0 > 0$   so that the time (at most{\footnote{See \cite{CM6} for the continuity of the time one flow; to get this also for the time at most one flow, replace the interior Schauder estimates in \cite{CM6} by the boundary estimate (see, e.g., theorem $4.29$ in \cite{L}).}}) one rescaled flow maps $B_{\delta_0}(\Sigma)$ to $B_{\frac{\epsilon_0}{2}} (\Sigma)$.  Since $\bar{M}_{\bar{T}}$ is close to $\Sigma$ up to translation and scaling,  lemma $7.10$ in \cite{CM1} and
  \eqr{e:Fclose} imply that $\bar{M}_{\bar{T}} \in B_{\delta_0}(\Sigma)$ if $\delta$ is small enough.  We have
  \begin{align}	
  	\bar{M}_t \in B_{\frac{\epsilon_0}{2}} (\Sigma) {\text{ for }} t \in [\bar{T} , \bar{T} + 1] \, .
  \end{align}
  Proposition \ref{p:Fsmall}, higher order interior estimates  and interpolation (cf. page $141$ of \cite{GT})  gives  $\bar{M}_t \in B_{{\epsilon_0}} (\Sigma) $  for $t\in[ \bar{T} + 1 , 0]$.
 \end{proof}

   \section{The proof of Proposition \ref{p:Fsmall}}   \label{s:s2}

We will need   expressions for geometric quantities for a graph $\Upsilon(u)$ of a function $u$ over   $\Sigma$.
We will assume that $|u|$ is small so  $\Upsilon (u)$ is contained in a tubular neighborhood of $\Sigma$ where the normal exponential map is invertible.   Let $e_{n+1}$ be the gradient of the signed distance function to $\Sigma$.
The geometric quantities that we need to compute on $\Upsilon (u)$   are:
\begin{itemize}
\item The relative area element $\nu_u (p) = \sqrt{\det g^u_{ij}(p)}/ \sqrt{\det g_{ij}(p)}$, where $g_{ij}(p)$ is the metric for $\Sigma$ at $p$ and
 $g^u_{ij}(p)$ is the pull-back metric from the graph of $u$ at $p$.
 \item The mean curvature $H_u(p)$ of $\Upsilon (u)$ at $(p+ u(p) \, \nn(p))$.
 \item The support function $\eta_u (p) = \langle p + u (p) \, \nn (p) , \nn_u \rangle$, where $\nn_u$ is the normal to $\Upsilon (u)$.
 \item The speed function $w_u (p) = \langle e_{n+1} , \nn_u \rangle^{-1}$ evaluated at the point $p + u (p) \, \nn(p)$.
\end{itemize}

\vskip2mm
The next lemma from \cite{CM3} (lemma $A.3$ there) gives the expressions for the $\nu_u$, $\eta_u$ and $w_u$ on a   graph $\Sigma_u$ over a general hypersurface
$\Sigma$:

\begin{Lem} \cite{CM3}	\label{l:areau}
There are   functions $w, \nu , \eta$ depending  on $(p , s , y)  \in \Sigma \times \RR \times T_p\Sigma$  that are smooth for $|s|$ sufficiently small and depend smoothly on $\Sigma$ so that:
\begin{itemize}
\item The speed function is given by $w_u (p) = w(p, u (p) , \nabla u (p))$.
\item   
The relative area element is given by $\nu_u (p) = \nu (p ,u (p) , \nabla u(p))$.  
\item The support function is given by $\eta_u (p) = \eta (p, u(p) , \nabla u (p))$.
\end{itemize}
In addition, the ratio $\frac{w}{\nu}$ depends only on $p$ and $s$.
Finally, the functions $w$, $\nu$, and $\eta$ satisfy:
\begin{itemize}
\item $w(p,s,0) \equiv 1$,  $\partial_s w(p,s,0) = 0$, $\partial_{y_{\alpha}}  w  (p,s,0) = 0$, and $ 
\partial_{y_{\alpha}} \partial_{  y_{\beta} }w (p,0,0) = \delta_{\alpha \beta}$.
\item  $\nu(p,0,0) =1$; the only non-zero first and second order terms are $ \partial_s \nu  (p,0,0) = H(p)$,
$ \partial_{p_j} \partial_s \nu  (p,0,0) = H_j(p)$,
$\partial_s^2 \nu(p,0,0) = H^2 (p) - |A|^2 (p)$, and 
$\partial_{y_{\alpha}} \partial_{  y_{\beta} }  \nu  (p,0,0) = \delta_{\alpha \beta}$.
\item $\eta (p,0, 0) = \langle p , \nn \rangle$, $ \partial_s  \eta  (p,0,0) = 1$, and $ \partial_{  y_{\alpha}} \eta (p,0,0) = - p_{\alpha}$.
\end{itemize}
\end{Lem}

Suppose now that $\Sigma \subset \RR^{n+1} $ is an    embedded shrinker and $u(p,t)$ is a  smooth function on $\Sigma \times (-\epsilon , \epsilon)$.  This gives a one-parameter family of hypersurfaces $\Upsilon (u)$.   Lemma $A.44$ in \cite{CM3} computes the graphical rescaled MCF equation:

\begin{Lem}	\label{l:normpart1}
The graphs $\Upsilon (u)$ satisfy the rescaled MCF equation if and only if $u$ satisfies
\begin{align}
	 \partial_t u(p,t)  =  w (p, u(p,t) , \nabla u (p,t)) \,  \left( \frac{1}{2} \, \eta (p, u(p,t) , \nabla u(p,t)) - H_u \right)   \equiv \cM \, u \, .
\end{align}
\end{Lem}

Using this, we can compute the linearization of $\cM$:
 
\begin{Cor}	\label{c:line}
The linearization of $\cM u$ at $u=0$ is given by
\begin{equation}
	\frac{d}{dr} \, \big|_{r=0} \, \cM (r\, u) = \Delta \, u + |A|^2 \, u - \frac{1}{2} \, \langle p , \nabla u \rangle + \frac{1}{2} \, u 
	= L \, u \, , 
\end{equation}
where $L$ is the second variation operator for the $F$ functional from section $4$ of 
\cite{CM1}.
\end{Cor}

  We will also use   the following elementary
calculus lemma:

\begin{Lem}	\label{l:calculus}
Let $\bar{U}$ be a $C^1$ function of $(p,s,y)$.  If $u$ and $v$ are $C^1$ functions on $\Sigma$, then 
\begin{align}
	\left| \bar{U} (p,u(p), \nabla u(p)) - \bar{U} (p,v(p), \nabla v(p)) \right| &\leq
	C_{\bar{U}} \, \left( \left| u(p) - v(p) \right|  +   \left| \nabla u (p) - \nabla v (p) \right| \right) \, ,
\end{align}
where $C_{\bar{U}} = \sup \{ \left| \partial_s \bar{U} \right| + \left| \partial_{y_{\alpha}} \bar{U} \right|
\, \big|   \, |s| + |y| \leq \| u \|_{C^1} + \| v \|_{C^1} \}$.
\end{Lem}

\subsection{Frechet differentiability}	\label{s:B}

In this subsection, we prove   Frechet differentiability of quasilinear elliptic operators.  This  is standard, but we were unable to
 locate a reference.  Recall   that if
 $X$ and $Y$ are Banach spaces and 
  $\Psi: X \to Y$ is continuous, then    $\Psi$ is Frechet differentiable at $x_0 \in X$ if there is a bounded linear map
  $A_{x_0}: X \to Y$ so that:
  \begin{enumerate}
  \item[] Given $\epsilon > 0$, there exists $\delta > 0$ so that if $|x|_{X} < \delta$, then 
  \begin{equation}
  	\left| \Psi (x_0 + x) - \Psi (x_0) - A_{x_0} \, x \right|_Y \leq \epsilon \, |x|_X \, .
  \end{equation}
  \end{enumerate}

 Throughout, $\Sigma$ is a closed manifold and $\cN$ is a quasilinear elliptic operator given by
 \begin{equation}	\label{e:cNapp}
 	\cN (u) = a_{ij} (p,u,\nabla u) \, u_{ij} + \Omega (p, u , \nabla u) \, , 
\end{equation}
where $a_{ij}$ and $\Omega$ are $C^3$ functions of $(p,s,y) \in \Sigma \times \RR \times T_p \Sigma$ and $a_{ij}$ is positive definite.
In particular,   $ \cN $ is continuous as a map from $C^{2,\alpha} \to C^{\alpha}$.

\begin{Lem}	\label{l:frechet}
$\cN$ is Frechet differentiable and the  derivative at $u$ is the linearized operator
\begin{align}
	L_u (v) &= a_{ij} (p,u,\nabla u) \, v_{ij} + \left( u_{ij} \,  (\partial_{y_{\alpha}} a_{ij})(p, u , \nabla u)\,  
	+  (\partial_{y_{\alpha}} \Omega)(p,u,\nabla u) \right) \, v_{\alpha} \notag \\
	&\qquad + \left( u_{ij} \,  
	(\partial_s a_{ij})(p, u , \nabla u) + (\partial_s \Omega) (p, u , \nabla u)  \right) \, v 
	\, .
\end{align}
\end{Lem}

We will use the following elementary fact in the proof:

\begin{Lem}	\label{l:holderfact}
If $f(p,s,y)$ is $C^3$ and has $f(p,0,0) = (\partial_s f)(p,0,0) = (\partial_{y_{\alpha}} f)(p,0,0)=0$, then there exists $C$ so that
for any $v \in C^{1,\alpha}$ we have
\begin{equation}
	\|  f(p, v(p) , \nabla v (p) \|_{C^{\alpha}} \leq C \, \| v \|_{C^1} \, \| v \|_{C^{1,\alpha}} \, .
\end{equation}
\end{Lem}

\begin{proof}
Since $f$ and its first derivative vanish at $(p,0,0)$, we have that
\begin{equation}	\label{e:holdc1b}
	\sup \{ \left| \partial_{p_i} f \right| + \left| \partial_s f \right| + \left| \partial_{y_{\alpha}} f \right|
\, \big|   \, |s| + |y| \leq   \Lambda \} \leq C \, \Lambda \, , 
\end{equation}
where $C$ depends on the bound for $\partial^2 f$ in this neighborhood.   
Lemma \ref{l:calculus}  gives 
\begin{align}	\label{e:holderfactC0}
	\left| f (p,v(p), \nabla v(p))   \right| &=
	\left| f (p,v(p), \nabla v(p)) - f (p,0, 0) \right| \leq
	C\, \| v \|_{C^1} \, \left( \left|   v(p) \right|  +   \left|  \nabla v (p) \right| \right) \, .
\end{align}
This gives the desired $C^0$ bound on $f$.  
To get the H\"older estimate, we will write
\begin{align}	\label{e:getHo}
	 &\frac{\left| f (p,v(p), \nabla v(p)) - f (q,v(q), \nabla v(q)) \right| }{ \dist^{\alpha}_{\Sigma} (p,q)} \leq 
	 \frac{ \left| f (p,v(p), \nabla v(p)) - f (q,v(p), \nabla v(p))  \right| }{ \dist^{\alpha}_{\Sigma} (p,q)} \notag \\
	 &\qquad
	 + \frac{ \left| f (q,v(p), \nabla v(p)) - f (q,v(q), \nabla v(q)) \right|	 }{ \dist^{\alpha}_{\Sigma} (p,q)}  \, .
\end{align}
 To estimate the last term, use the derivative
estimate \eqr{e:holdc1b} for $f$ to get
\begin{align}
	   \frac{ \left| f (q,v(p), \nabla v(p)) - f (q,v(q), \nabla v(q)) \right|	 }{ \dist^{\alpha}_{\Sigma} (p,q)}
	 &\leq 	\frac{ C \, \| v \|_{C^1} \, \left(   |v(p) - v(q)| +
	|\nabla v(p) - \nabla v(q)| \right) }{ \dist^{\alpha}_{\Sigma} (p,q)} \notag \\
	&  \leq  C \, \| v \|_{C^1} \, \| v \|_{C^{1,\alpha}}  
	 	\, .
\end{align}
 To bound the first term on the right in \eqr{e:getHo}, we 
need one more observation.  Namely, since
$f(p,0,0) = (\partial_s f)(p,0,0) = (\partial_{y_{\alpha}} f)(p,0,0) \equiv 0$, we can differentiate in $p$ 
 to see that 
 \begin{equation}
 	(\partial_s \partial_{p_i} f)(p,0,0) = (\partial_{y_{\alpha}}\partial_{p_i} f)(p,0,0) \equiv 0 \, .
 \end{equation}
 In particular, we get that
 \begin{equation}
	\sup \{ \left| \partial_s \partial_{p_i}  f \right|  + \left| \partial_{y_{\alpha}} \partial_{p_i}  f \right|
\, \big|   \, |s| + |y| \leq   \| v \|_{C^1} \} \leq C \, \| v \|_{C^1} \, , 
\end{equation}
where the constant depends on the $C^3$ norm of $f$.  Therefore, 
 Lemma \ref{l:calculus}  can be applied  to get for $|s| + |y| \leq \| v \|_{C^1}$ that 
$
 	\left| (\partial_{p_i} f )(p,s, y)   \right| \leq C \, \| v \|^2_{C^1} $.
	 Using this  and the earlier $C^0$ estimate \eqr{e:holderfactC0}, we 
   bound the remaining term in  \eqr{e:getHo}  
\begin{align}
	   \frac{ \left| f (p,v(p), \nabla v(p)) - f (q,v(p), \nabla v(p))  \right| }{ \dist^{\alpha}_{\Sigma} (p,q)}
	   &\leq 	 2 \, \| f \|^{1-\alpha}_{C^0} \left(
	    \frac{ \left| f (p,v(p), \nabla v(p)) - f (q,v(p), \nabla v(p))  \right|}{ \dist_{\Sigma} (p,q)}
	    \right)^{\alpha} \notag \\
	    &\leq 	 C \, \| v \|_{C^1}^{2 - 2 \alpha} \, \left(  \| v \|^2_{C^1}  \right)^{\alpha} 
	  =  C \, \| v \|^2_{C^1}  
	 	\, .
\end{align} 
\end{proof}

\begin{proof}[Proof of Lemma \ref{l:frechet}]
To simplify notation, each function will be evaluated at $p$ or $(p,u, \nabla u)$ unless otherwise specified.
Differentiating $\cN$ at $u$ in the direction of $v$ gives
\begin{align}
	L_ u \, v &\equiv \frac{d}{dt} \big|_{t=0} \, \left\{ a_{ij} (p,u+t v ,\nabla u + t \nabla v) \, (u_{ij} +t \, v_{ij})
	 + \Omega (p, u+ t \, v  , \nabla u + t \, \nabla v)
	\right\} \notag \\
	&= a_{ij}  \, v_{ij} + u_{ij}  \left(
	(\partial_s a_{ij}) \, v + (\partial_{y_{\alpha}} a_{ij}) \, v_{\alpha} \right) 
	  + (\partial_s \Omega)   \, v + (\partial_{y_{\alpha}} \Omega)  \, v_{\alpha} \, .
\end{align}
To prove the lemma, we will  show that  the error $ \cN (u+v) - \cN (u) - L_u \, v$ is quadratic in $v$.
It is convenient to divide the error into three parts
\begin{align}	\label{e:threeerror}
	\cN (u+v) - \cN (u) - L_u \, v =  \Omega_u (p, v , \nabla v) + b_{ij} (p, v , \nabla v) \, u_{ij}  
	+ c_{ij} (p, v , \nabla v) \, v_{ij} \, , 
\end{align}
where  $\Omega_u  $,  $b_{ij} $ 
   and $c_{ij}   $ are given by
\begin{align}
	\Omega_u (p, v , \nabla v) & = \Omega (p, u+v , \nabla u + \nabla v) - \Omega  - 
	 (\partial_s \Omega)  \, v    - (\partial_{y_{\alpha}} \Omega)  \, v_{\alpha}  \, , \\
	 b_{ij} (p,v, \nabla v) & = a_{ij} (p,u+  v ,\nabla u +   \nabla v)   - 
	 a_{ij}  
	  -  
	(\partial_s a_{ij}) \, v   - (\partial_{y_{\alpha}} a_{ij}) \, v_{\alpha}  \, , \\
	c_{ij}  (p, v , \nabla v)  & =  a_{ij} (p,u+  v ,\nabla u +   \nabla v) - a_{ij}  
	\, .
\end{align}
Observe that $\Omega_u (p, v , \nabla v)$ and  $ b_{ij} (p,v, \nabla v) $ satisfy the hypotheses of Lemma \ref{l:holderfact}, so we have
\begin{align}	\label{e:mfft1}
	\| \Omega_u (p, v , \nabla v) \|_{C^{\alpha}} +	\| b_{ij} (p,v, \nabla v)  \|_{C^{\alpha}} &\leq C \, \| v \|_{C^1} \, \| v \|_{C^{1,\alpha}} \, .
\end{align}
Therefore, using the ``$C^{\alpha}$ Leibniz rule'' gives
\begin{align}	\label{e:mfft2}
	\| b_{ij} (p,v, \nabla v) \, u_{ij}  \|_{C^{\alpha}} &\leq \| b_{ij} (p,v, \nabla v)  \|_{C^{\alpha}} \, \| u_{ij} \|_{C^0} + 
	\| b_{ij} (p,v, \nabla v)  \|_{C^{0}} \, \| u_{ij} \|_{C^{\alpha}} \notag \\
	&\leq 
	C  \, \| u \|_{C^{2,\alpha}} \, \| v \|_{C^1} \, \| v \|_{C^{1,\alpha}}  \, .
\end{align}
To bound the third term in \eqr{e:threeerror}, observe that $c_{ij} (p,0,0) = 0$ so Lemma \ref{l:calculus} gives
\begin{equation}	\label{e:cij0}
	\left|  c_{ij} (p,s, y) \right| \leq C \, \left( |s| + |y| \right) \, .
\end{equation}
Moveover, differentiating $c_{ij} (p,0,0) \equiv 0$ implies that $(\partial_{p_k} c_{ij} )(p,0,0) = 0$ so 
Lemma \ref{l:calculus} gives
\begin{equation}	
	\left|  (\partial_{p_k} c_{ij} ) (p,s, y) \right| \leq C \, \left( |s| + |y| \right) \, .
\end{equation}
We can use this, \eqr{e:cij0} and the fact that $c_{ij}$ is $C^1$ in $s$ and $y$ to get
\begin{align}
	 &\frac{\left| c_{ij} (p,v(p), \nabla v(p)) - c_{ij} (q,v(q), \nabla v(q)) \right| }{ \dist^{\alpha}_{\Sigma} (p,q)} \leq 
	 \frac{ \left| c_{ij} (p,v(p), \nabla v(p)) - c_{ij} (q,v(p), \nabla v(p))  \right| }{ \dist^{\alpha}_{\Sigma} (p,q)} \notag \\
	 &\qquad
	 + \frac{ \left| c_{ij} (q,v(p), \nabla v(p)) - c_{ij}(q,v(q), \nabla v(q)) \right|	 }{ \dist^{\alpha}_{\Sigma} (p,q)}  \leq 2 \, \|c_{ij}\|^{1-\alpha}_{C^0} \,  
	 \| \partial_{p_k} c_{ij} \|^{\alpha}_{C^0} + C \, \| v \|_{C^{1,\alpha}}
	 \leq C \, \| v \|_{C^{1,\alpha}}   \notag \, .
\end{align}
It follows that 
$
	\| c_{ij} (p,v(p), \nabla v(p)) \|_{C^{\alpha}} \leq C \, \| v \|_{C^{1,\alpha}} $ and, thus,   the ``$C^{\alpha}$ Leibniz rule'' gives
\begin{align}	\label{e:mfft3}
	\| c_{ij} (p,v, \nabla v) \, v_{ij}  \|_{C^{\alpha}} &\leq  
	C  \, \| v \|_{C^{2,\alpha}} \,    \| v \|_{C^{1,\alpha}}  \, .
\end{align}
Finally, combining \eqr{e:mfft1}, \eqr{e:mfft2} and \eqr{e:mfft3} gives that the error is quadratic in $v$
\begin{equation}
	\| \cN (u+v) - \cN (u) - L_u \, v \|_{C^{\alpha}} \leq C\,  \| v \|_{C^{1,\alpha}}  \, \| v \|_{C^{2,\alpha}}  \, .
\end{equation}
\end{proof}

\subsection{Proof of Proposition \ref{p:Fsmall}}
 \vskip4mm
  Suppose now that $\tilde M_t$ flows by the rescaled MCF for $t_1 \leq t \leq t_2$ and each $\tilde M_t$ is in a neighborhood of a fixed closed shrinker $\Sigma$.  We have that $F(t)$ is non-increasing and applying Proposition \ref{t:loja} with $\Gamma = \tilde M_t$ gives
\begin{align}	\label{e:loja2a}
  	\left| F(t) - F(\Sigma) \right|^{2 -  \beta } \leq 
	\int_{\tilde M_t} \left| \frac{ \langle x , \nn \rangle}{2} - H \right|^2 \,  \e^{ - \frac{|x|^2}{4} } = \left| F' (t) \right| \, ,
  \end{align}
  where $\beta \in (0,1)$.
  We will need the following   simple lemma:

    \begin{Lem}	\label{l:calclem1}
Let $G:[0, T] \to \RR$ be non-negative with  $\left| G(t)  \right|^{2 -  \beta  } \leq  
	 | G'(t)|$ for some  $\beta \in (0,1)$.  If $G'\leq 0$, then 
  	$\left| G (t) \right|  \leq   \left( G^{\beta -1} (0) + (1-\beta) t \right)^{- \frac{1}{1-\beta} } \, .	$
On the other hand, if $G' \geq 0$, then
  	$ \left| G (t) \right|  \leq   \left( G^{\beta -1} (T) + (1-\beta) (T-t) \right)^{- \frac{1}{1-\beta} } \, .	$
  \end{Lem}

\begin{proof}
Suppose first that $G'\leq 0$, so 
we get the differential inequality
\begin{equation}
	\left( G^{ \beta - 1} \right)' = (\beta - 1) \, G^{\beta - 2} \, G' \geq  1-\beta  > 0 \, .
\end{equation}
Integrating this from $0$ to $t$ gives the first claim.   The second claim follows from integrating the opposite differential inequality from $t$ to $T$.
\end{proof}

  Suppose that on $[0,s]$ we have that $F(t)\geq F(\Sigma)$ and on $[s,T]$ we have that $F(\Sigma)\geq F(t)$.  Either interval is allowed to be empty.   Set $G(t)=F(t)-F(s)$ on $[0,s]$ and set $G(t)=F(s)-F(t)$ on $[s,T]$.  In each case,  $G\geq 0$ and, by \eqr{e:loja2a}, $\left| G(t)  \right|^{2 -  \beta  } \leq  
	 | G'(t)|$.  On $[0,s]$, $G'\leq 0$, while $G' \geq 0$  on $[s,T]$.  
 On   $[0,s]$, Lemma \ref{l:calclem1} gives
 \begin{align}	\label{e:fromcalci1}
 	[( F(0)-F(s))^{\beta-1}+(1-\beta)\,t]^{-\frac{1}{1-\beta}} &\geq  F(t) - F(s) = 
 	  - \int_t^{s} F'(r) \,dr\\
	  &=   \int_t^{s} \int_{M_r} \left| \frac{ \langle x , \nn \rangle}{2} - H \right|^2 \, \e^{ - \frac{|x|^2}{4} }   \, dr  \, ,\notag
\end{align}
whereas on the second interval $[s,T]$  we get
\begin{align}	\label{e:fromcalci2}
 	[( F(s)-F(T))^{\beta-1}+(1-\beta)\,(T-t)]^{-\frac{1}{1-\beta}} &\geq     \int_{s}^t \int_{M_r} \left| \frac{ \langle x , \nn \rangle}{2} - H \right|^2 \, \e^{ - \frac{|x|^2}{4} }   \, dr  \, . 
\end{align}

We   use these bounds to prove uniform integral bounds,   independent of the time interval.  
We will need the following simple estimate for geometric series:

\begin{Lem}	\label{l:geometric}
If  $\beta \in (0,1)$,  $  \gamma \in (1 ,  (1-\beta)^{-1})$ and $c_1 > 0$, then 
\begin{align}
\sum_{j=1}^{\infty}  2^{ \gamma j}  \left( c_1  +  2^{j+1} \right)^{- \frac{1}{1-\beta} } \leq 
2  \, \left(  \frac{1}{1-\beta} - \gamma \right) \, \left( 2 + c_1 \right)^{ \gamma - \frac{1}{1-\beta}  } 
 \, .
\end{align}
\end{Lem}

\begin{proof}
For each $j \geq 1$, we have that 
\begin{align}
	2^{ \gamma j}  \left( c_1  +  2^{j+1} \right)^{- \frac{1}{1-\beta} }  
	&\leq 
\min_{[2^j,2^{j+1}]} \, \, \left\{  r^{\gamma} (c_1 +r)^{- \frac{1}{1-\beta} } \right\}  \leq 2^{-j}\, \int_{2^j}^{2^{j+1}}  
r^{\gamma} (c_1 +r)^{- \frac{1}{1-\beta} }  \, dr \notag \\
&\leq  2  \, \int_{2^j}^{2^{j+1}}  r^{\gamma - 1} (c_1 +r)^{- \frac{1}{1-\beta} }  \, dr  \leq 2 \, \int_{2^j}^{2^{j+1}}   (c_1 +r)^{\gamma - 1 - \frac{1}{1-\beta} }  \, dr \, , \notag 
\end{align}
where the last inequality used that $r^{\gamma - 1} \leq (c_1 + r)^{\gamma - 1}$ (since $\gamma > 1$).  Summing  
over $j$ gives
\begin{align}
\sum_{j=1}^{\infty}  2^{ \gamma j}  \left( c_1  +  2^{j+1} \right)^{- \frac{1}{1-\beta} } &\leq  
2  \, \int_{2}^{\infty}   (c_1 +r)^{\gamma - 1 - \frac{1}{1-\beta} }  \, dr =
2  \, \left(  \frac{1}{1-\beta} - \gamma \right) \, \left( 2 + c_1 \right)^{\gamma - \frac{1}{1-\beta}  } \, . \notag
\end{align}
\end{proof}

\begin{Lem}	\label{l:wint1}
With the assumptions of Lemma \ref{l:geometric}, 
 there exists $C$ so that
\begin{align}
	\int_1^{s} r^{  \gamma} \,  \int_{M_r} \left| \frac{ \langle x , \nn \rangle}{2} - H \right|^2 \, \e^{ - \frac{|x|^2}{4} }   \, dr &\leq C \,   \left( F(0) - F(s) \right)^{1 - \gamma \, (1-\beta)} \, ,  \label{e:first247} \\
	\int_s^{T-1} (T-r)^{  \gamma} \,  \int_{M_r} \left| \frac{ \langle x , \nn \rangle}{2} - H \right|^2 \, \e^{ - \frac{|x|^2}{4} }   \, dr &\leq C \,   \left( F(s) - F(T) \right)^{1 - \gamma \, (1-\beta)} \, . \label{e:second247}
\end{align}
\end{Lem}

\begin{proof}
 For \eqr{e:first247}, 
divide the integral into dyadic parts and use the bound \eqr{e:fromcalci1} to get
\begin{align}	\label{e:gmag}
	\int_1^{s} r^{\gamma} \,  \int_{M_r} \left| \frac{ \langle x , \nn \rangle}{2} - H \right|^2 \, \e^{ - \frac{|x|^2}{4} }   \, dr &= \sum_{j=1}^{\infty} \int_{2^{j-1} }^{\max \{ s , 2^j\} } r^{ \gamma} \,  \int_{M_r} \left| \frac{ \langle x , \nn \rangle}{2} - H \right|^2 \, \e^{ - \frac{|x|^2}{4} }   \, dr  \notag \\
	&\leq   \sum_{j=1}^{\infty}  2^{ \gamma j} \int_{2^{j-1}}^{s}  \,  \int_{M_r} \left| \frac{ \langle x , \nn \rangle}{2} - H \right|^2 \, \e^{ - \frac{|x|^2}{4} }   \, dr   \\
	&\leq  \sum_{j=1}^{\infty}  2^{ \gamma j}  \left( ( F(0)-F(s))^{\beta-1} + (1-\beta) \, 2^{j-1} \right)^{- \frac{1}{1-\beta} } 
	\, . \notag
\end{align}
Set $c_1 = \left(  \frac{4}{1-\beta} \right) \, ( F(0)-F(s))^{\beta-1}$ and 
$c_2 = \left(  \frac{1-\beta}{4} \right)^{- \frac{1}{1-\beta} }$ and
 rewrite the bound in  \eqr{e:gmag} as
 \begin{align}
 	\int_1^{s} r^{\gamma} \,  \int_{M_r} \left| \frac{ \langle x , \nn \rangle}{2} - H \right|^2 \, \e^{ - \frac{|x|^2}{4} }   \, dr
	&\leq c_2 \, 
	\sum_{j=1}^{\infty}  2^{ \gamma j}  \left( c_1  +  2^{j+1} \right)^{- \frac{1}{1-\beta} } \, .
\end{align}
Since $ 1 < \gamma <  \frac{1}{1- \beta } $ and $c_1 >  0$, we can apply Lemma 
\ref{l:geometric} to get
\begin{align}
\sum_{j=1}^{\infty}  2^{ \gamma j}  \left( c_1  +  2^{j+1} \right)^{- \frac{1}{1-\beta} } \leq 
2  \, \left(  \frac{1}{1-\beta} - \gamma \right) \, \left( 2 + c_1 \right)^{ \gamma - \frac{1}{1-\beta}  } 
 \, .
\end{align}
 To complete the proof of \eqr{e:first247}, note that $\gamma - \frac{1}{1-\beta}  < 0$ so we have
 \begin{align}
 	 \left( 2 + c_1 \right)^{ \gamma - \frac{1}{1-\beta}  } & \leq 
	   c_1^{ \gamma - \frac{1}{1-\beta}  }   = \left(  \frac{4}{1-\beta} \right)^{ \gamma - \frac{1}{1-\beta}  } \,
	   ( F(0)-F(s))^{(\beta-1)\,\left(  \gamma - \frac{1}{1-\beta} \right)  }  \, .
 \end{align}
 The second bound \eqr{e:second247} follows similarly, except that we replace $r$ by $T-r$ and use the bound
 \eqr{e:fromcalci2} in place of  \eqr{e:fromcalci1}. 
\end{proof}

\begin{Lem}	\label{l:dist}
There exist $C$ and an open neighborhood $\cO_{\Sigma}$ of $\Sigma$ so that if 
  $\tilde M_t \subset \cO$ is a graphical solution of rescaled MCF on $[t_1 , t_2 ]$, then
\begin{align}
	\int_{\Sigma} \left| u (p, t_2) - u (p,t_1) \right| \leq   C\, \int_{t_1}^{t_2} \int_{\tilde M_r} \left| \frac{ \langle x , \nn \rangle}{2} - H \right| \, \e^{ - \frac{|x|^2}{4} }   \, dr\, .
\end{align}
\end{Lem}

\begin{proof}   
It follows from Lemma \ref{l:normpart1}  that each $\tilde M_t$ is given as the graph of a function $u(p,t)$ over $\Sigma$, where $|u|$ and $|\nabla u|$ are small and $u$ satisfies 
\begin{align}
	 \partial_t u(p,t)  =  w (p, u(p,t) , \nabla u (p,t)) \,  \left( \frac{1}{2} \, \eta (p, u(p,t) , \nabla u(p,t)) - H_u \right)   \equiv \cM \, u \, .
\end{align}
 Lemma \ref{l:areau} gives that $w$ is uniformly bounded, so we get
 \begin{equation}	\label{e:utbound}
 	\left| \partial_t u(p,t) \right|   \leq C   \,  \left| \frac{1}{2} \, \eta (p, u(p,t) , \nabla u(p,t)) - H_u \right| 
	= C\, \left| \frac{ \langle x , \nn \rangle}{2} - H \right| \, ,
 \end{equation}
 where the last expression is evaluated at the point $p + u(p,t) \, \nn_{\Sigma} (p)$ on the hypersurface $\tilde M_t$. 
 
Since $|u|$ and $|\nabla u|$ are small,  Lemma \ref{l:areau} gives that the area elements on $\Sigma$ and $\tilde M_t$ are uniformly equivalent
 (the ratio of these area elements is   $\nu_u$).  Therefore, \eqr{e:utbound} gives 
  \begin{align}	\label{e:utbound2}
 	\int_{\Sigma} \left| \partial_t u(p,t) \right|   &\leq C   \, \int_{\Sigma}  \left| \frac{1}{2} \, \eta (p, u(p,t) , \nabla u(p,t)) - H_u \right| \notag \\
	&\leq C'   \, \int_{\Sigma}  \left| \frac{1}{2} \, \eta (p, u(p,t) , \nabla u(p,t)) - H_u \right| \, \nu_u 
	=  C' \, \int_{\tilde M_t} \left| \frac{ \langle x , \nn \rangle}{2} - H \right| \, .
 \end{align}
Using the fundamental theorem of calculus and    Fubini's theorem gives
   \begin{align}	\label{e:utbound3}
 	\int_{\Sigma} \left| u (p, t_2) - u (p,t_1) \right|  &\leq \int_{\Sigma} \left( \int_{t_1}^{t_2} \left| \partial_t u (p, t ) \right| \, dt \right)
	=  \int_{t_1}^{t_2}  \left( 
	\int_{\Sigma} \left| \partial_t u(p,t) \right| \right) \, dt \notag \\
	 &\leq  C \,  \int_{t_1}^{t_2}  \int_{\tilde M_t} \left| \frac{ \langle x , \nn \rangle}{2} - H \right| \, dt 
	 \, .
 \end{align}
The lemma follows from this since $\Sigma$ is compact and, thus, each $M_t$ lies in a bounded set where
 $ \e^{ - \frac{|x|^2}{4} } $ has a positive lower bound.

\end{proof}

  \begin{proof}[Proof of Proposition \ref{p:Fsmall}]
 Lemma \ref{l:dist}   gives  that
 \begin{align}
 	\int_{\Sigma} \left| u (p, t_2) - u (p,t_1) \right| &\leq C\, 
	  \int_{t_1}^{t_2}  \int_{\tilde M_r} \left| \frac{ \langle x , \nn \rangle}{2} - H \right| \, \e^{ - \frac{|x|^2}{4} }   \, dr \, .
\end{align}
As above,  suppose that $F(t) \geq F(\Sigma)$ on $[t_1 , s]$ and $F(t) \leq F (\Sigma)$ on $[s, t_2]$.  We will divide the integral on the right into   integrals over the four subintervals $[t_1 , t_1 + 1]$, $[t_1 +1 , s]$, $[s, t_2 -1]$ and $[t_2 - 1 , 1]$.
The first  is   bounded by Cauchy-Schwarz and     \eqr{e:fromcalci1}
\begin{align}
	 \left( \int_{t_1}^{t_1 + 1}  \int_{\tilde M_r} \left| \frac{ \langle x , \nn \rangle}{2} - H \right| \, \e^{ - \frac{|x|^2}{4} }   \, dr \right)^2 &\leq
	  \int_{t_1}^{t_1 + 1}  \left( \int_{\tilde M_r} \left| \frac{ \langle x , \nn \rangle}{2} - H \right|  \, \e^{ - \frac{|x|^2}{4} }  \right)^2 \, dr \notag \\
	  &\leq   \int_{t_1}^{t_1 + 1} F(r) \,  \int_{\tilde M_r} \left| \frac{ \langle x , \nn \rangle}{2} - H \right|^2  \, \e^{ - \frac{|x|^2}{4} }    \, dr \\
	  &\leq F(t_1) \, \left( F(t_1) - F(t_1 + 1) \right) \, .  \notag
\end{align}
The last is bounded similarly  
\begin{align}
	 \left( \int_{t_2 - 1}^{t_2}  \int_{\tilde M_r} \left| \frac{ \langle x , \nn \rangle}{2} - H \right| \, \e^{ - \frac{|x|^2}{4} }   \, dr \right)^2  \leq
	 F(t_2 -1) \, \left( F(t_1) - F(t_2) \right) \, .
\end{align}
To bound the second, set   $\gamma = \frac{1}{2} \,  (1 + (1-\beta)^{-1})$,  use Cauchy-Schwarz and Lemma \ref{l:wint1} to get
\begin{align}
	 &\left( \int_{t_1+1}^{s}  \int_{\tilde M_r} \left| \frac{ \langle x , \nn \rangle}{2} - H \right| \, \e^{ - \frac{|x|^2}{4} }   \, dr \right)^2 \notag \\
	 &\qquad \qquad \leq
	 \left(
	 \int_{t_1+1}^{s} (r-t_1)^{-\gamma} \, dr \right) \int_{t_1+1}^{s}  
	  (r-t_1)^{\gamma} \left( \int_{\tilde M_r} \left| \frac{ \langle x , \nn \rangle}{2} - H \right|  \, \e^{ - \frac{|x|^2}{4} }  \right)^2 \, dr \notag \\
	  &\qquad \qquad   \leq
	  \left(
	 \int_{1}^{\infty} r^{-\gamma} \, dr \right) F(t_1  ) \,  \int_{t_1+1}^{s}  
	  (r-t_1)^{\gamma}   \int_{\tilde M_r} \left| \frac{ \langle x , \nn \rangle}{2} - H \right|^2  \, \e^{ - \frac{|x|^2}{4} }    \, dr \\
	  &\qquad \qquad \leq C \, F(t_1) \, \left( F(t_1 ) - F(s) \right)^{1 - \gamma \, (1-\beta)}
	   \, , \notag
\end{align}
where $C$ depends only on   $\beta$.   
Arguing similarly, we bound the third by
\begin{align}
	 \left( \int_{s}^{t_2 -1}  \int_{\tilde M_r} \left| \frac{ \langle x , \nn \rangle}{2} - H \right| \, \e^{ - \frac{|x|^2}{4} }   \, dr \right)^2  
	   \leq C \, F(t_1) \, \left( F(s) - F(t_2 ) \right)^{1 - \gamma \, (1-\beta)}
	   \, .
\end{align}
Finally, since we have an upper bound for $F(t_1)$ and we can assume that $|F(t_1) - F(t_2)| \leq 1$, combining these four bounds gives the proposition with $\beta = \frac{1}{2}\,
\left(1 - \gamma \, (1-\beta)\right) = \frac{\beta}{4}$.
 \end{proof}


 \appendix

 \section{The Lojaciewicz-Simon inequality}   \label{s:s3}
 
  The classical Lojaciewicz inequality, e.g. \cite{CM5}, is about analytic functions on Euclidean space.  It asserts that near a critical point $x$ of an analytic function $f:\RR^n\to \RR$
    \begin{align}
    	\left| f(x) - f(y) \right|^{2 - {\beta} } \leq |\nabla f (y)|^2 \, ,
\end{align}
where $\beta \in( 0,1)$ is a constant.   We will need Schulze's, \cite{Sc}, Lojaciewicz-Simon inequality:

\begin{Pro}	\label{t:loja}
 (Lojaciewicz-Simon inequality for $F$, \cite{Sc}.)  
 If $\Sigma$ is a closed shrinker and $\beta \in (0,1)$, then there exists $\epsilon> 0$, $\beta \in (0,1)$ 
  so that if  
 $\| u \|_{C^{2,\alpha}} \leq \epsilon$, then
 \begin{equation}
 	\left| F(\Upsilon (u)) - F(\Sigma) \right|^{2- \beta} \leq \left| \nabla F (\Upsilon (u)) \right|^2 
	\equiv \int_{\Upsilon (u)} \left| H - \frac{\langle x , \nn \rangle}{2} \right|^2 \, \e^{ - \frac{|x|^2}{4} } \, .
 \end{equation}
\end{Pro}

 Proposition \ref{t:loja} will be a consequence of a general Lojaciewicz-Simon inequality, Lemma \ref{t:lojsim} below, that
 relies on work of Leon Simon, \cite{S}.  We 
  organize the argument in a way that is useful for future reference and include   additional useful information not covered elsewhere.

Let $E$ be the  space of $C^{2,\alpha}$ functions on $\Sigma$ and 
 $F$  an analytic functional on $E$  
\begin{equation}
	F(u) = \int_{\Sigma} G(x, u (x) , \nabla u(x)) \, d\mu_x \, ,
\end{equation}
where $G$ is a positive analytic function of $(x,s,y)$ for $x \in \Sigma$, $s \in \RR$, and $y \in T_x \Sigma$.  
Let
$Q$ be the  positive definite  symmetric
bilinear form 
\begin{equation}
	Q (u,v) = \int_{\Sigma} u(x) \, v(x) \, G (x,0,0) \, d\mu_x \, .
\end{equation}
Note that $Q$ is continuous on $E$.   Let $\| u \|_Q = \sqrt{Q (u,u)}$   denote the associated norm.

We have two important operators from $E$ to $C^{0,\alpha}$: The   nonlinear Euler-Lagrange operator $\cN \, (u)$, i.e.,   the negative of the gradient of $F$,  and its linearization $L$ at $0$. These are 
\begin{align}	\label{e:defcN}
	Q (\cN (u) , \phi ) &=  - \frac{d}{ds} \big|_{s=0} \, F( u+ s \, \phi)  {\text{ for every $u$ and $\phi$ in $C^2 (\Sigma)$,}}\\
  	L (v) &=  \frac{d}{ds} \big|_{s=0} \, \cN ( s\, v) \, .
  \end{align}
 In particular, $u$ is a critical point for $F$ if and only if $\cN\, (u)=0$ and the `tangent space' to the set of critical points for $F$ at $u=0$ is contained in the kernel of $L$.

\begin{Lem}	\label{t:lojsim}
There exist $\beta \in (0,1)$ and a neighborhood $\cO \subset E$ of $0$ so that for $u \in \cO$
\begin{equation}	\label{e:Ljo}
	\left| F(u) - F(0) \right|^{2 - \beta }  \leq \| \cN (u) \|_{Q}^2 \, .
\end{equation}
\end{Lem}

\vskip2mm
To explain the idea, let
$
	\cK  
$
 be the kernel of $L$ and $\Pi $ the $Q$-orthogonal projection to $\cK$.  By elliptic theory, $\cK$ is finite dimensional.  
 One extreme case of   \eqr{e:Ljo} is where we restrict $F$ to a finite dimensional space.  This case  follows from the classical Lojaciewicz inequality.  
 At the other extreme,    $L$ is invertible and \eqr{e:Ljo} follows from Taylor expansion of $F$ and  does not require analyticity.  
 The general case uses the invertible case and Lyapunov-Schmidt   reduction to reduce to the classical  finite dimensional  case.

Define a map $\bar{\cN} : E \to C^{0,\alpha}$ by
$
	\bar{\cN}   = \Pi   + \cN  $.
	We will  show that near $0$ this map is analytic, one to one and onto, so the inverse function theorem  gives an inverse $\Psi$.
	Note  that  $\Psi (\cK)$  contains the critical points of $F$.  
The construction of   $\Psi$ is the Lyapunov-Schmidt   method, producing a  finite dimensional analytic submanifold containing  the  critical points of $F$.   

The next lemma constructs    $\Psi$  and establishes   its basic properties.  It is useful to introduce  a weighted $W^{2,2}$ norm $Q_2$ on $E$  
\begin{equation}
	\| u \|^2_{Q_2}  = \int_{\Sigma} \left( u^2 + |\nabla u|^2 + \left| \Hess_u \right|^2 \right) \, G(x,0,0) \, d\mu_x \, .
\end{equation}
We will need some properties  that will be proven for the $F$-functional in subsection \ref{ss:FLoj}:
  \begin{enumerate}
  \item[(N1)] We have   $\| u \|_{E} \leq C \, \left( \| u \|_{C^0} + \| \cN  \, u  \|_{C^{0,\alpha}} \right)$ and  
  $\| u \|_{Q_2} \leq C \, \left( \| u \|_Q + \| \cN \, u \|_{Q} \right)$ for $u$ in a neighborhood of $0$.
   \item[(N2)] $L$ is the 
   Frechet derivative of $\cN$ at $0$.
      \item[(N3)]  
      If 
  $\| u \|_{E} + \| v \|_{E} \leq C_1$, then 
$
 	\| \cN (u) - \cN (v) \|_{Q} \leq C_2 \,   \| u-v \|_{Q_2}   
$ for $C_2 = C_2 (C_1)$.
 \item[(N4)]   $\cN$ is analytic in a neighborhood of $0$.
 \end{enumerate}

\begin{Lem}	\label{l:linear}
There exists $\delta > 0$ and an inverse mapping $\Psi : B_{\delta}(0) \subset
C^{0,\alpha} \to E$   with $\Psi \circ \bar{\cN} (u) = \bar{\cN} \circ \Psi (u) = u$ and
\begin{enumerate}
\item $\Psi$ is bounded from $Q$ to $Q_2$; in particular, also from $Q$ to $Q$.
\item $\Psi$ is Lipschitz from 
$Q$ to $Q_2$, i.e., 
$\| \Psi (u) - \Psi ( v) \|_{Q_2} \leq C \,   \|  u - v \|_{Q} $.
\item The function $f: \cK \to \RR$ defined by $f(u) = F(\Psi (u))$ is analytic.
\end{enumerate}
\end{Lem}

\begin{proof} 
 By (N2), 
  $\Pi   + L  $  is
  the Frechet derivative of $\bar{\cN}$ at $0$. It follows from (N1) that $\Pi + L$ is bounded   from 
 $Q_2$ to $Q$, bounded  from $C^{2,\alpha}$ to $C^{0,\alpha}$, and $Q$-self-adjoint.
Since $\Pi  + L  $ has trivial kernel and $\bar{\cN}$ is analytic by (N4),  
the analytic inverse function theorem   gives $\delta > 0$ and  $\Psi$ (cf. section $2.7$ in \cite{N}).  Properties (1), (2) and (3) follow.
\end{proof}

\begin{Lem}	\label{l:F}
There exists $C$ so that for every sufficiently small $u \in E$
\begin{equation}
	 \left| F (u) - f(\Pi (u)) \right| \leq C \, \| \cN (u)  \|^2_{Q} \, .
\end{equation}
\end{Lem}

\begin{proof}
 Define 
the  family of functions $	u_t = u +t \, \left( \Psi (\Pi \, u) - u \right) = u + t\,  \Psi \cN (u )$.
The definition of $f$,   fundamental theorem of calculus and   ``first variation formula'' give 
\begin{align}
	\left| F (u) - f(\Pi u)  \right| \leq  	   \int_0^1 \left|  \frac{d}{dt} \, F (u_t) \right| \, dt   =   \int_0^1  \left| Q\left( \cN (u_t) , \Psi \cN (u )  \right) \right|
	  \, dt \leq  \|  \Psi  \cN (u)   \|_Q \, \int_0^1 \, \|  \cN (u_t) \|_Q  \, dt  \, . \notag
\end{align}
The first term on the right is bounded by $C \, \| \cN (u) \|_Q$ since $\Psi$ is bounded on $Q$ by (1).    To bound the second, use (N3)  and property (1)  of $\Psi$ give that
to get
\begin{align}
	\| \cN (u_t)  \|_{Q} &\leq C \, \left( \| \cN (u) \|_{Q} + \| \Psi \, \cN (u) \|_{Q_2} \right)
	\leq  C \, \left( \| \cN (u) \|_{Q} + C' \,  \| \cN (u) \|_{Q} \right)    \, . \notag
\end{align}
\end{proof}

\begin{Lem}	\label{l:grad}
There exists $C$ so that for every sufficiently small $u \in E$
\begin{equation}
	|\nabla_{\cK} f | (\Pi (u)) \leq C \, \| \cN (u)  \|_{Q} \, .
\end{equation}
\end{Lem}

\begin{proof}
If $w , v \in \cK$, then    the ``first variation formula'' for $f=F \circ \Psi$ gives
\begin{equation}
	\frac{d}{ds} \big|_{s=0} \, f( w + s \, v) = \frac{d}{ds} \big|_{s=0} \, F ( \Psi \left( w + s \, v \right) )
	= - Q\left(  \cN ( \Psi (w)) \, , \left\{   
	\frac{d}{ds} \big|_{s=0} \,   \Psi \left( w + s \, v \right) 
	\right\} \right)  \, .  \notag
\end{equation}
It follows from the Lipschitz property of $\Psi$, i.e., property (2), that 
\begin{align}
	\left\| \frac{d}{ds} \big|_{s=0} \,   \Psi \left( w + s \, v \right) \right\|_{Q} \leq C \, \| v \|_{Q}  \, .
\end{align}
Therefore,  we have
$ 
	|\nabla_{\cK} f| (w) \leq C \, \| \cN \circ \Psi (w) \|_{Q}$ when $w \in \cK$ and, thus, for any $u$ 
\begin{align}	\label{e:step1of1}
	\left|\nabla_{\cK} f \right| (\Pi (u)) \leq C \, \| \cN \circ \Psi \circ \Pi (u) \|_{Q} \, .
\end{align}
This is close to what we want, except that $\cN$ is evaluated at $\Psi \circ \Pi (u)$ instead of at $u$.
The definition of $\Psi$ gives   $\Psi \circ \Pi (u) = u - \Psi \circ \cN (u)$, so 
(N3) and property (1)  give 
\begin{align}	 
	  \| \cN \circ \Psi \circ \Pi (u) \|_{Q} &=   \| \cN \left(  u - \Psi  \cN u
	  \right)  \|_{Q}
	   \leq C \, \left(   \| \cN  (u) \|_{Q} + \|  \Psi  \cN (u) \|_{Q_2} \right)  \leq \bar{C} \, 
	    \| \cN  (u) \|_{Q}
	  \, .  \notag
\end{align}
\end{proof}

\begin{proof}[Proof of Lemma \ref{t:lojsim}]
  Lemma \ref{l:grad} and the classical Lojasiewicz inequality for  $f$ on $\cK$ give
\begin{align}	\label{e:step1}
	C \, \| \cN (u)  \|_{Q} \geq \left|\nabla_{\cK} f \right| (\Pi (u)) \geq \left| f (\Pi (u)) - f(0) \right|^{1 - \frac{\beta}{2} } \, .
\end{align}
On the other hand, since $F(0) = f(0)$,  Lemma \ref{l:F} gives
\begin{align}
	\left| f (\Pi (u)) - f(0) \right| &\geq \left| F(u) - F(0) \right| -
	 \left| f (\Pi (u)) - F(u) \right|  \geq	\left| F(u) - F(0) \right| -  C \, \| \cN (u)  \|^2_{Q} \, . \notag
\end{align}
Combining these two inequalities gives the desired estimate.
\end{proof}

\subsection{The required properties for $F$}  \label{ss:FLoj}

If we set
  $F(u) = F(\Upsilon (u))$, then $F$ is a functional on $E$ with the map $G: \Sigma \times \RR \times T\Sigma \to \RR$ given by
\begin{align}	\label{e:defGforF}
	G(p,s,y) = \e^{- \frac{|p + s \, \nn (p)|^2}{4} } \, \nu (p,s, y)  \, ,
\end{align}
where $\nu$ is given by Lemma \ref{l:areau}.  We will show in Corollary \ref{c:cN} that Lemma \ref{t:lojsim} applies and, thus, complete the proof of Proposition \ref{t:loja}.

\begin{Lem}
$G$ is uniformly analytic in $s$ and $y$.  Namely, there is $\beta > 0$ so that if $|z|, |w|, |p|, |q| < \beta$, then $G$ can be expanded in a power series
\begin{align}
	G(x,z+\lambda_1 w , p + \lambda_2 \, q) = 
	\sum_{|\alpha| \geq 0} \, G_{\alpha} (x,z,w,p,q) \, \lambda^{\alpha} \, , 
\end{align}
where for all $|\lambda| < 1$ we have   
$
	  \left| \sum_{|\alpha|=j} \, 
	G_{\alpha} (x,z,w,p,q) \, \lambda^{\alpha} \right|  \leq 1  
$ with $|z|, |w|, |p|, |q| < \beta$.
\end{Lem}

\begin{proof}
This follows immediately since the product of analytic functions is also analytic, $\nu$ is analytic by construction in  \cite{CM3}, and
$ \e^{- \frac{|p + s \, \nn (p)|^2}{4} }$ is also analytic.
\end{proof}

Let  $\cM$ be  from Lemma \ref{l:normpart1}, so   $\partial_t u = \cM u$ is the graphical rescaled MCF equation.

\begin{Lem}	\label{l:fvar}
The operator $\cN$ is given by
$	\cN (u) = \zeta (p, u , \nabla u)   \, \cM (u) $,
where  $\zeta (p,s,y)$  is a smooth function with
$
	\zeta (p,0,0) = 1$.   
\end{Lem}

\begin{proof}
The first variation formula (see, e.g., lemma $3.1$ in \cite{CM1})  gives
\begin{align}
	\frac{d}{ds} \, \big|_{s=0} \, F( u+ sv) = \int_{\Upsilon (u)} v \, 
	\langle e_{n+1} , \nn_u \rangle \, \left( H_u - \frac{ \langle x , \nn_u \rangle    }{2}  \right) \,  \e^{- \frac{|x|^2}{4} } \, .
\end{align}
We convert this to an integral on $\Sigma$, introducing the relative volume element $\nu_u$, to get
\begin{align}
	\frac{d}{ds} \, \big|_{s=0} \, F( u+ sv) = \int_{\Sigma}  
	\frac{v}{w_u}  \, \left( H_u - \frac{\eta_u}{2}  \right) \,  \e^{- \frac{|p + u  \, \nn  |^2}{4} } \, \nu_u \, ,
\end{align}
where  $w_u$, $H_u$, $\eta_u$ are from Lemma \ref{l:areau}.  The definitions of $\cN$, $Q$ and $G$ give
\begin{align}
	\frac{d}{ds} \, \big|_{s=0} \, F( u+ sv)  = -Q ( \cN (u) , v) = - \int_{\Sigma} 
	\cN (u) \, v \, G(p, 0,0) = - \int_{\Sigma} 
	\cN (u) \, v \, \e^{ - \frac{|p|^2}{4} } \, .
\end{align}
Thus, equating the two expressions gives 
\begin{equation}
	\cN (u) = \frac{\nu_u}{w_u}  \, \left(   \frac{\eta_u}{2} - H_u \right) \,  \e^{ \frac{|p|^2 - |p + u  \, \nn  |^2}{4} } 
	=  \frac{\nu_u}{w^2_u}  \,   \e^{ \frac{|p|^2 - |p + u  \, \nn  |^2}{4} } \, \cM (u) \, .
\end{equation}
The lemma follows from this and Lemma \ref{l:areau}.
\end{proof}

 \begin{Cor}	\label{c:cN}
The linearization of $\cN$ at $0$ is    $L$ and (N1), (N2), (N3), (N4) hold.
 \end{Cor}
 
 \begin{proof}
By Lemma  \ref{l:fvar},
$
	\cN (u) =   \zeta (p, u , \nabla u)   \, \cM (u) 
$.   
Since $\cM (0)= 0$ and $\zeta (p,0,0) = 1$,  the linearizations of $\cN$ and $\cM$ agree at $0$, so the first claim follows by Corollary \ref{c:line}.
To get the first part of (N1), observe that $\cN (u)$ can be written as 
\begin{equation}	\label{e:togetN1}
	\cN (u) = a_{ij} (p,u,\nabla u) \, u_{ij} + \Omega (p, u , \nabla u) \, , 
\end{equation}
 where
$a_{ij}$, $\Omega$ are smooth, 
$a_{ij} (p,0,0) = \delta_{ij}$ and $\Omega (p,0,0) = 0$.  
This last condition gives 
\begin{equation}
	 \| \Omega \|_{C^{\alpha}} \leq C \,  \| u \|_{C^{1, \alpha}} \, .
\end{equation}
Since $u$ is assumed to be small in  $C^{2,\alpha}$, we get that $a_{ij}$ has a  uniform $C^{\alpha}$ bound.  Thus, we can apply linear Schauder estimates to get that
\begin{equation}
	\| u \|_{C^{2,\alpha} } \leq C \, \left( \| u \|_{C^0} + \| \Omega \|_{C^{\alpha}} + \| \cN (u) \|_{C^{\alpha}} \right) 
	\leq C' \left( \| u \|_{C^0} + \| u \|_{C^{1,\alpha}}  + \| \cN (u) \|_{C^{\alpha}} \right)  \, .
\end{equation}
To complete the bound, use interpolation  (page $141$ of \cite{GT})
 to absorb the $\| u \|_{C^{1,\alpha}} $ term.

 For the second part of (N1), we use \eqr{e:togetN1} and   linear  $W^{2,2}$ estimates   to  get
 \begin{align}	\label{e:N1g}
 	\| u \|_{W^{2,2}} \leq C \, \left( \| \cN (u) \|_{L^2} + \| \Omega (p, u , \nabla u) \|_{L^2} 
	\right) \, .
 \end{align}
 Since $\Omega (p,0,0) = 0$, the fundamental theorem of calculus gives
$
 	\left| \Omega (p,u,\nabla u) \right| \leq C\, \left( |u| + |\nabla u| \right) 
$ and, thus, 
 $ \| \Omega (p, u , \nabla u) \|_{L^2}  \leq C \, \| u \|_{W^{1,2}}$. The $W^{2,2}$ estimate in (N1) follows by using this in \eqr{e:N1g} and then using
   interpolation  (page $173$ of \cite{GT})    $\| u \|_{W^{1,2}} \leq C_{\epsilon}\, \| u \|_{L^2}  + \epsilon \, \| u \|_{W^{2,2}}$.

 The property (N2) follows from Lemma \ref{l:frechet} since  the linearization at $u = 0$ is $L$.
 To get   (N3), we will use   use the form \eqr{e:togetN1} of the equation to write
\begin{align}	 
	\cN (u) - \cN (v) = a_{ij} (p,u,\nabla u) \, u_{ij} &+ \Omega (p, u , \nabla u) 
	-a_{ij} (p,v,\nabla v) \, v_{ij} - \Omega (p, v , \nabla v)  \notag \\
	 =  \Omega (p, u , \nabla u) - \Omega (p, v , \nabla v)  +&
	 a_{ij} (p,u,\nabla u) \, (u_{ij} - v_{ij} )   +  \left( a_{ij} (p,u,\nabla u) - a_{ij} (p,v,\nabla v)
	 \right) \, v_{ij} 
	\, . \notag
\end{align}
 Lemma \ref{l:calculus} (the fundamental theorem of calculus) and the  $C^2$ bound for $u$ gives
 \begin{align}	\label{e:togetN1b}
	\left| \cN (u) - \cN (v) \right| &\leq 
	C\,  \left( |u-v| + \left| \nabla u - \nabla v \right| \right) +
	 C \, \left| \Hess_u - \Hess_v  \right| 
	\, .  
\end{align}
Property (N3) follows by squaring  and integrating.
The analyticity of $\cN$, i.e.,  (N4), 
follows similarly.
 \end{proof}

\end{document}